\documentclass{amsart}

\usepackage{amssymb,amsmath,color}
\usepackage{amsthm}
\usepackage{url}
\usepackage{tikz}
\usepackage{hyperref}

\definecolor{red}{rgb}{1,0,0}
\definecolor{blue}{rgb}{.2,.2,.8}

\def\pre{\textrm{pre}}

\newtheorem{theorem}{Theorem}
\newtheorem{corollary}[theorem]{Corollary}
\newtheorem{proposition}{Proposition}
\newtheorem{conjecture}{Conjecture}

\newtheorem{lemma}[theorem]{Lemma}

\theoremstyle{definition}

\newtheorem{remark}{Remark}

\begin{document}

\allowdisplaybreaks

\title[Partitions and elementary symmetric polynomials]{Partitions and elementary symmetric polynomials - an experimental approach}
\author{Cristina Ballantine}\address{Department of Mathematics and Computer Science\\ College of the Holy Cross \\ Worcester, MA 01610, USA \\} 
\email{cballant@holycross.edu} 
\author{George Beck} \address{Department of Mathematics and Statistics\\ Dalhousie University\\ Halifax, NS, B3H 4R2, Canada \\} \email{george.beck@gmail.com} 
\author{Mircea Merca}\address{Department of Mathematical Methods and Models\\
 Fundamental Sciences Applied in Engineering Research Center\\  University Politehnica of Bucharest\\
 RO-060042 Bucharest, Romania\\}\email{mircea.merca@upb.ro}

\begin{abstract}
 Given a partition $\lambda$,  we write $e_j(\lambda)$ for the $j^{\textrm{th}}$ elementary symmetric polynomial $e_j$ evaluated at the parts of $\lambda$ and  $e_jp_A(n)$ for the sum of $e_j(\lambda)$ as $\lambda$ ranges over the set of partitions of $n$ with parts in $A$. For $e_jp_A(n)$, we prove analogs of the classical formula for the partition function, $p(n)=1/n \sum_{k=0}^{n-1}\sigma_1(n-k)p(k)$, where $\sigma_1$ is the sum of divisors function.  We  prove several congruences for $e_2p_4(n)$, the sum of $e_2$ over  the set of partitions of $n$ into four parts. Define the function $\pre_j(\lambda)$ to be the multiset of monomials in $e_j(\lambda)$, which is itself a partition. If $\mathcal A$ is a set of partitions, we define $\pre_j(\mathcal A)$ to be the set of partitions $\pre_j(\lambda)$ as $\lambda$ ranges over $\mathcal A$. If $\mathcal P(n)$ is the set of all partitions of $n$, we conjecture that  the number of odd partitions in  $\pre_2(\mathcal P(n))$ is at least the number of distinct partitions.  We prove some results about $\pre_2(\mathcal B(n))$, where $\mathcal B(n)$  is the set of binary partitions of $n$. We conclude with  conjectures on the log-concavity of functions related to $e_jp(n)$, the sum of $e_j(\lambda)$ for all $\lambda\in \mathcal P(n)$.
\end{abstract}

\maketitle
\section{Introduction}

An integer partition $\lambda$ of a positive integer $n$ is a weakly decreasing sequence of positive integers $\lambda_i$ whose sum is $n$. We use the notation $\lambda=(\lambda_1, \lambda_2, \ldots, \lambda_\ell)$ with $\lambda_1\geq  \lambda_2\geq \ldots \geq \lambda_\ell>0$ and $|\lambda|:=\sum_{i=1}^\ell\lambda_i=n$. We refer to $|\lambda|$ as the size of $\lambda$ and the numbers $\lambda_i$ as the parts of $\lambda$.
 The length of $\lambda$ is the number of parts of $\lambda$ and is denoted by $\ell(\lambda)$. We denote by $\mathcal P(n)$ the set of partitions of $n$ and let $p(n):=|\mathcal P(n)|$. Since the empty partition is the only partition of $0$, we have $p(0)=1$. We use the convention that $p(x)=0$ if $x\not\in \mathbb Z_{\geq 0}$ (and similarly for other partition statistics). For more on the theory of partitions, we refer the reader to \cite{Andrews98}.

Different classes of symmetric polynomials defined in terms of partitions form bases for the graded algebra of symmetric functions defined over the rational numbers. For details, we refer the reader to \cite{Stanley}. Recently, the third author has derived several partition identities as specializations of the fundamental
relations between complete and elementary symmetric functions (see \cite{M}).

In this article, we take a  different approach and  study properties of elementary symmetric polynomials evaluated at parts of partitions.  Recall that the $j^{\textrm{th}}$ elementary symmetric polynomial is defined by $$e_j(X_1, X_2, \ldots,X_n)=\begin{cases}\displaystyle\sum_{1\leq i_1<i_2<\cdots<i_j\leq n}X_{i_1}X_{i_2}\cdots X_{i_j} & \text{ if } j\leq n, \\ 0 & \text{ if } j> n.\end{cases} $$ Thus, given a partition $\lambda$, we have $e_j(\lambda)=0$ if $\ell(\lambda)<j$ and otherwise $$e_j(\lambda)=\sum_{1\leq i_{1}<i_{2}<\cdots< i_{j}\leq \ell(\lambda)}\lambda_{i_{1}}\lambda_{i_{2}}\cdots \lambda_{i_{j}}.$$

In \cite{SchSi}, the authors study the case $e_j(\lambda)$ where $\ell(\lambda)=j$ and call this statistic the norm of $\lambda$.

Given a set $A$ of positive integers, we denote by $\mathcal P(n|A)$ the set of partitions of $n$ with parts from $A$. We set $p(n|A):=|\mathcal P(n|A)|$.  We introduce the function  $$e_jp_A(n):=\sum_{\lambda\in \mathcal P(n|A)}e_j(\lambda)$$  and define the restricted  divisor function 
$$
\sigma_{j,A}(n) := \sum_{\substack{d|n\\d\in A}} d^j.
$$
Our first result establishes relations between the partition function $e_jp_A(n)$ and the divisor function $\sigma_{j,A}(n)$ when $j=2$ or $3$. 
\begin{theorem}\label{Th1}
For $n>0$,
\begin{enumerate}
    \item [(1)] $\displaystyle{e_2p_A(n) =\frac{1}{2} \left(n^2\,p(n|A)-\sum_{k=1}^n \sigma_{2,A}(k)\,p(n-k|A) \right) }$;
    \item [(2)] $\displaystyle{e_3p_A(n) =\frac{1}{6} \left(n^3\,p(n|A)-\sum_{k=1}^n \big( 3\,n\,\sigma_{2,A}(k)-2\,\sigma_{3,A}(k) \big)\,p(n-k|A) \right)}$.
\end{enumerate}
\end{theorem}
Then we consider  $\mathcal P(m,n)$, the set of partitions of $n$ with length equal to  $m$ and set 
$p_m(n):=|\mathcal P(m,n)|$ and $$e_jp_m(n):=\sum_{\lambda\in \mathcal P(m,n)} e_j(\lambda).$$
In \cite{NW} Nijenhuis and Wilf establish the minimal periods for $p_m(n)$ modulo primes and in \cite{K} Kwong establishes the minimal periods for $p_m(n)$ modulo prime powers. Kronholm \cite{Kr} gives Ramanujan-like congruences and other general congruence properties for $p_m(n)$.

There is considerable numerical evidence to conjecture that the function $e_2p_4(n)$  satisfies the linear recurrence
\begin{align*}y(22+n) =& -y(n)-y(1+n)+3 y(3+n)+6 y(4+n)+3 y(5+n)-3 y(6+n)\\ & -12 y(7+n) -12 y(8+n)-2 y(9+n)+10 y(10+n)+18 y(11+n)\\ & +10 y(12+n)-2 y(13+n)-12 y(14+n)-12
y(15+n)-3 y(16+n)\\ & +3 y(17+n)+6 y(18+n)+3 y(19+n)-y(21+n)\end{align*} and therefore $e_2p_4(n)$ would have periods modulo $m$ for positive $m$ and sufficiently large $n$.

Below we give a conjectural table of minimal periods for certain values of $m$. 
$$\begin{array}{c|cccccccccccccc}m & 2& 3& 4& 5 & 6& 7& 8 & 9&10&11&12&13&14     \\ \hline \text{period} & 48 & 54& 96& 300 & 432& 84& 192& 324& 1150& 132& 864 & 156& 336
\end{array}$$
$$\begin{array}{c|ccccccccccccc}m &16 &  17& 18& 19& 21& 22& 23& 26& 27& 28& 29& 31& 37 \\ \hline \text{period} & 384& 204& 1246 & 228& 756 & 528 & 276 & 624 & 972 & 672 & 348 & 372 & 444
\end{array}$$

 We have used Mathematica\texttrademark \,  here and for many other calculations.

We prove the first three entries in the table above. 
\begin{theorem}\label{Th0}
  For $n\geq 0$,
  $$e_2p_4(n+48) \equiv e_2p_4(n) \pmod 2,$$
  $$e_2p_4(n+54) \equiv e_2p_4(n) \pmod 3,$$
  $$e_2p_4(n+96) \equiv e_2p_4(n) \pmod 4.$$
\end{theorem}

For comparison, from \cite[Theorem 2]{NW}, the minimal period for $p_4(n)$ modulo $2$ is $24$ and modulo $3$ it is $36$. From \cite[Theorem 14]{K}, the minimal period for $p_4(n)$ modulo $4$ is $48.$

 We conclude this section by describing some notation used throughout the paper.  We identify a partition with its multiset of parts and  write  $\lambda\cup \mu$ for the partition obtained by taking the multiset union of parts of $\lambda$ and $\mu$. Similarly, if   $\mu$ is a sub-multiset of $\lambda$, we write $\lambda\setminus \mu$ for the partition obtained by taking the multiset difference of parts of $\lambda$ and $\mu$. We denote by $m_\lambda(i)$ the number of times $i$ appears as a part in $\lambda$.  If $\lambda$ is a partition with parts divisible by $k$, we denote by $\mu/k$ the partition obtained by dividing each part of $\mu$ by $k$.

\section{Proof of Theorem \ref{Th1}}

First, we prove a useful lemma. 
Given a function $f$ defined on  subset $A$ of non-negative integers, let
     $$S_{A,f}(n):=\sum_{\lambda\in\mathcal P(n|A)} \sum_{i=1}^{\ell(\lambda)} f(\lambda_i).$$
   \begin{lemma} The generating function for $S_{A,f}(n)$ is given by 
   \begin{equation} \label{gf_S}
   \sum_{n=0}^\infty S_{A,f}(n)\,q^n= \prod_{a\in A} \frac{1}{1-q^a} \sum_{a\in A} \frac{f(a)\,q^a}{1-q^a}.
   \end{equation} 
   \end{lemma}
   \begin{proof} 
     \begin{align*}
        \sum_{n=0}^\infty S_{A,f}(n)\,q^n 
        &= \frac{\partial}{\partial z}  \Bigg|_{z=1}\prod_{a\in A} (1+z^{f(a)}\,q^a + z^{2f(a)}\,q^{a+a}+ z^{3f(a)}\,q^{a+a+a}+\cdots)  \\
        &= \frac{\partial}{\partial z} \prod_{a\in A} \frac{1}{1-z^{f(a)\,q^a}} \Bigg|_{z=1} 
        = \prod_{a\in A} \frac{1}{1-q^a} \sum_{a\in A} \frac{f(a)\,q^a}{1-q^a}.
     \end{align*}\end{proof}

\begin{proof}[Proof of Theorem \ref{Th1}]

As given in \cite{MS}, the generating function for  $\sigma_{j,A}(n)$ is
$$\sum_{n=1}^\infty \sigma_{j,A}(n)\,q^n = \sum_{a\in A} \frac{a^j\,q^a}{1-q^a}.$$

    For a partition $\lambda=(\lambda_1,\lambda_2,\ldots,\lambda_\ell)$ and a positive  integer $k$ we define the power sum $P_k(\lambda)$ as
    $$P_k(\lambda) := \sum_{i=1}^\ell \lambda_i^k.$$
    From the relations
    $$e_2(\lambda) = \frac{1}{2} \left( P_1(\lambda)^2 - P_2(\lambda) \right)$$
    and
    $$e_3(\lambda) = \frac{1}{6} \left( P_1(\lambda)^3 - 3\,P_1(\lambda)\,P_2(\lambda)+2\,P_3(\lambda) \right),$$
    we can write
    \begin{align*}
        e_2p_A(n) & = \sum_{\lambda\in\mathcal{P}(n|A) } e_2(\lambda)
        = \frac{1}{2} \sum_{\lambda\in\mathcal{P}(n|A)}  \left( P_1(\lambda)^2 - P_2(\lambda) \right)\\
        & = \frac{1}{2} \left(n^2\,p(n) - \sum_{\lambda\in\mathcal{P}(n|A)} P_2(\lambda) \right)
    \end{align*}   
    and
    \begin{align*}
        e_3p_A(n) & = \sum_{\lambda\in\mathcal{P}(n|A) } e_3(\lambda)
        = \frac{1}{6} \sum_{\lambda\in\mathcal{P}(n|A)}  \left( P_1(\lambda)^3 - 3\,P_1(\lambda)\,P_2(\lambda)+2\,P_3(\lambda) \right)\\
        & = \frac{1}{6} \left(n^3\,p(n|A) - 3\,n\sum_{\lambda\in\mathcal{P}(n|A)} P_2(\lambda) + 2 \sum_{\lambda\in\mathcal{P}(n|A)} P_3(\lambda) \right).
    \end{align*}

    From \eqref{gf_S} with  $f(a)=a^j$, we deduce 
    \begin{align*}
    \sum_{n=1}^\infty \sigma_{j,A}(n)\,q^n = \sum_{a\in A} \frac{a^j\,q^a}{1-q^a} 
    = \prod_{a\in A} (1-q^a) \sum_{n=0}^\infty \left( \sum_{\lambda\in\mathcal{P}(n|A)} P_j(\lambda) \right) q^n.
    \end{align*}
        Thus we can write
    \begin{align*}
    \sum_{n=0}^\infty \left( \sum_{\lambda\in\mathcal{P}(n|A)} P_j(\lambda) \right) q^n
    & = \prod_{a\in A} \frac{1}{1-q^a} \sum_{n=1}^\infty \sigma_{j,A}(n)\,q^n   \\
    & = \left( \sum_{n=0}^\infty p(n|A)\,q^n \right) \left( \sum_{n=1}^\infty \sigma_{j,A}(n)\,q^n \right).
    \end{align*}
    Multiplying the two power series and comparing coefficients, we obtain
    $$
    \sum_{\lambda\in\mathcal{P}(n|A)} P_j(\lambda) = \sum_{k=1}^n \sigma_{j,A}(k)\,p(n-k|A).
    $$
    This concludes the proof.
\end{proof}

\subsection{ Particular cases of Theorem \ref{Th1}}

First, let $A$ be the set of all positive integers.  We write $e_jp(n)$ for $e_jp_A(n)$.  We have the following consequence of Theorem \ref{Th1}. 
\begin{corollary}\label{C1}
For $n>0$,
\begin{enumerate}
    \item [(i)] $\displaystyle{e_2p(n)=\frac{1}{2} \left(n^2\,p(n)-\sum_{k=1}^n \sigma_2(k)\,p(n-k) \right)};$
    \item [(ii)] $\displaystyle{e_3p(n)=\frac{1}{6} \left(n^3\,p(n)-\sum_{k=1}^n \big( 3\,n\,\sigma_2(k)-2\,\sigma_3(k) \big)\,p(n-k) \right)}$.
\end{enumerate}
\end{corollary}

Next, let $A$ be the set of all odd positive integers.
We write $Q(n)$ for $p(n|A)$ and   $e_jQ(n)$ for $e_jp_A(n)$.   We have the following consequence of Theorem \ref{Th1},
where   $$\sigma_{j,\mathrm{odd}}(n) = \sum_{\substack{d|n\\ d \textrm{ odd}}} d^j.$$

\begin{corollary}\label{C2}
For $n>0$,
\begin{enumerate}
    \item [(i)] $\displaystyle{e_2Q(n)=\frac{1}{2} \left(n^2\,Q(n)-\sum_{k=1}^n \sigma_{2,\mathrm{odd}}(k)\,Q(n-k) \right)};$
    \item [(ii)] $\displaystyle{e_3Q(n)=\frac{1}{6} \left(n^3\,Q(n)-\sum_{k=1}^n \big( 3\,n\,\sigma_{2,\mathrm{odd}}(k)-2\,\sigma_{3,\mathrm{odd}}(k) \big)\,Q(n-k) \right)}$.
\end{enumerate}
\end{corollary}

Finally, let $A$ be the set of non-negative integer powers of $2$.  
Recall that a binary partition of $n$ is a partition with all parts powers of $2$. We write $B(n)$ for $p(n|A)$ and $e_jB(n)$ for $e_jp_A(n)$.  We have the following consequence of Theorem \ref{Th1},
where $$\sigma_{j,\mathrm{bin}}(n) = \sum_{2^d|n} 2^{d\, j}.$$ 

\begin{corollary}\label{C7}
For $n>0$,
\begin{enumerate}
    \item [(i)] $\displaystyle{e_2B(n)=\frac{1}{2} \left(n^2\,B(n)-\sum_{k=1}^n \sigma_{2,\mathrm{bin}}(k)\,B(n-k) \right)};$
    \item [(ii)] $\displaystyle{e_3B(n)=\frac{1}{6} \left(n^3\,B(n)-\sum_{k=1}^n \big( 3\,n\,\sigma_{2,\mathrm{bin}}(k)-2\,\sigma_{3,\mathrm{bin}}(k) \big)\,B(n-k) \right)}$.
\end{enumerate}
\end{corollary}

Next, we determine the parity of $e_jB(n)$.
\begin{theorem}\label{Th3} Let $n,j \geq 2$ be integers. Then, $$e_jB(n)\equiv \binom{n-2}{j-2}\pmod 2.$$
\end{theorem}

\begin{proof}

  Let $\lambda\in \mathcal B(n)$. 
  
    If $m_\lambda(1)<j$, then $e_j(\lambda)\equiv 0 \pmod 2$ since all summands in $e_j(\lambda)$ are even. 

    If $m_\lambda(1)\geq j$, then $$e_j(\lambda)\equiv \binom{m_\lambda(1)}{j}\pmod 2.$$

    Next, notice that $m_\lambda(1)=k$ if and only if $(\lambda\setminus(1^k))/2\in \mathcal B\left(\frac{n-k}{2}\right).$  Hence, $$e_jB(n)\equiv \sum_{k=j}^n \binom{k}{j}B\left(\frac{n-k}{2}\right)\pmod 2.$$

    It can be shown by induction that $B(n)\equiv 0\pmod 2$ for $n\geq 2$. Moreover $B(0)=B(1)=1$. Thus, 
    $$e_jB(n)\equiv  \binom{n}{j} + \binom{n-2}{j}\pmod 2.$$ Using Pascal's formula,  
    \begin{align*}
        \binom{n}{j} + \binom{n-2}{j}& = \binom{n-1}{j-1} +\binom{n-1}{j}+ \binom{n-2}{j}\\ & =\binom{n-1}{j-1} +\binom{n-2}{j-1}+ \binom{n-2}{j}+ \binom{n-2}{j} \\ & \equiv \binom{n-2}{j-2}\pmod 2. 
    \end{align*}

For the last equivalence we used Pascal's formula again: $$\binom{n-1}{j-1} -\binom{n-2}{j-1}=\binom{n-2}{j-2}.$$
\end{proof}
We single out the following two cases of Theorem \ref{Th3}.

\begin{corollary}\label{C9}
    For $n\geq 2$,
    $$ e_2B(n) \equiv 1 \pmod 2.$$
\end{corollary}

\begin{corollary}
    For $n\geq 2$,
    $$ e_3B(n) \equiv n \pmod 2.$$
\end{corollary}
Since $\sigma_{2,\textrm{bin}}(k)\equiv 1\pmod 2$ if $k\geq 1$ and $B(0)=B(1)=1$, from Corollary \ref{C9} and  Corollary \ref{C7}(i), we see that,  for  $n\geq 3$,   $$n^2\,B(n)-\sum_{k=1}^{n-2} B(n-k)\equiv 0\pmod 4.$$ \begin{corollary} Let $n\geq 3$. If $n$ is even $$\sum_{k=1}^{n-2}B(n-k)\equiv 0\pmod 4,$$ and  if $n$ is odd $$B(n)-\sum_{k=1}^{n-2}B(n-k)\equiv 0\pmod 4.$$\end{corollary}

\section{Proof of Theorem \ref{Th0}}

\noindent \textit{Proof of ${e_2p_4(n)\equiv e_2p_4(n+48)\pmod 2}$.}

For two formal power series, we write $$\sum_{n=0}^\infty a(n)q^n \equiv \sum_{n=0}^\infty b(n)q^n \pmod m$$ to mean $a(n)\equiv b(n)\pmod m$ for all $n\geq 0$. 
 
  By the definition of $e_2(\lambda)$, if $\lambda\in \mathcal P(4,n)$, then  $e_2(\lambda)$ is odd if and only if $\lambda$ has one or two even parts.
  We first establish the generating function for the number of partitions of $n$ into $4$ parts with exactly two even parts.

Recall that the conjugate of the partition $\lambda=(\lambda_1, \lambda_2, \ldots, \lambda_\ell)$ is the partition $\lambda'$ with $\lambda'_j=|\{i\mid \lambda_i\geq j\}|$. Considering conjugates of partitions, the generating function for the number of partitions of $n$ into two  even parts is given by  $$\frac{q^4}{1-q^4}\frac{1}{1-q^2}$$ and the generating function for the number of partitions of $n$ into two odd parts is given by $$\frac{q^2}{1-q^4}\frac{1}{1-q^2}.$$
 
 Thus, the generating function for the number of partitions of $n$ into $4$ parts with exactly two even parts is \begin{align*} \frac{q^6}{(1-q^2)^2(1-q^4)^2}. \end{align*}

Similarly,  the generating function for the number of partitions of $n$ into four parts with exactly one part even is given  by $$\frac{q^2}{1-q^2}\cdot\frac{q^3}{(1-q^2)(1-q^4)(1-q^6)}=\frac{q^5}{(1-q^2)^2(1-q^4)(1-q^6)}.$$

Thus, \begin{align*}\sum_{n=0}^\infty e_2p_4(n)q^n& \equiv\frac{q^6}{(1-q^2)^2(1-q^4)^2}+\frac{q^5}{(1-q^2)^2(1-q^4)(1-q^6)} \pmod 2\\ & \equiv 
\left(\frac{q^6}{1-q^4}+\frac{q^5}{1-q^6}\right)\frac{1}{(1-q^2)^4} \pmod 2\\ & \equiv 
\left(\frac{q^6}{1-q^4}+\frac{q^5}{1-q^6}\right)\frac{1}{1-q^8} \pmod 2.\end{align*}

 Then \begin{align*}\notag \sum_{n=0}^\infty (e_2p_4(n)-e_2p_4(n-48))q^n &  \equiv  \left(\frac{q^6}{1-q^4}+\frac{q^5}{1-q^6}\right)\frac{1-q^{48}}{1-q^8}\pmod 2 \\ &  \equiv \left(\frac{q^6}{1-q^4}+\frac{q^5}{1-q^6}\right)A(q) \pmod 2, 
\end{align*} 
where $A(q):=1+q^8+q^{16}+q^{24}+q^{32}+q^{40}$.

Writing  $A(q)=(1+q^{8})(1+q^{16}+q^{32})$, we have 
\begin{align*}\frac{q^6}{1-q^4}A(q)& \equiv \frac{q^6}{1-q^4}(1-q^{8})(1+q^{16}+q^{32})\pmod 2\\ & = q^6(1+q^4)(1+q^{16}+q^{32}).\end{align*}

Writing  $A(q)=(1+q^{24})(1+q^{8}+q^{16})$, we have  
\begin{align*}\frac{q^5}{1-q^6}A(q)& \equiv \frac{q^5}{1-q^6}(1-q^{24})(1+q^{8}+q^{16})\pmod 2\\ & = q^5(1+q^{12})(1+q^6)(1+q^{8}+q^{16}). \end{align*} Hence, for $n\geq 48$, $e_2p_4(n)-e_2p_4(n-48)\equiv 0 \pmod 2$.

\medskip

\noindent \textit{Proof of ${e_2p_4(n)\equiv e_2p_4(n+54)\pmod 3}$.}

By listing all combinations of residues modulo $3$ for parts of $\lambda\in \mathcal P(4,n)$, we see that 
\begin{itemize}

\item[(i)] $e_2(\lambda)\equiv 1\pmod 3$ if and only if there are two different residues modulo $3$ among the parts of $\lambda$ and there are two parts congruent to each such residue; 

\item[(ii)] $e_2(\lambda)\equiv 2\pmod 3$ if and only if there are three different residues modulo $3$ among the parts of $\lambda$ (thus, there are two parts congruent to each other modulo $3$ and the other two parts have residues different from each other and from the pair).

\item[(iii)]  $e_2(\lambda)\equiv 0\pmod 3$ in all other cases.
\end{itemize}

The generating function for the number of partitions of $n$ in case (i) is given by \begin{align*}\frac{q^6}{1-q^6}\frac{1}{1-q^3}\frac{q^2}{1-q^6}\frac{1}{1-q^3} & +\frac{q^6}{1-q^6}\frac{1}{1-q^3}\frac{q^4}{1-q^6}\frac{1}{1-q^3}\\ & +\frac{q^2}{1-q^6}\frac{1}{1-q^3}\frac{q^4}{1-q^6}\frac{1}{1-q^3}\\ & = \frac{q^{10}+q^8+q^6}{(1-q^3)^2(1-q^6)^2}.  
\end{align*}

The generating function for the number of partitions of $n$ in case (ii) is given by \begin{align*}\frac{q^6}{(1-q^6)(1-q^3)}\frac{q}{1-q^3}\frac{q^2}{1-q^3}& +\frac{q^2}{(1-q^6)(1-q^3)}\frac{q^3}{1-q^3}\frac{q^2}{1-q^3}\\ & +\frac{q^4}{(1-q^6)(1-q^3)}\frac{q^3}{1-q^3}\frac{q}{1-q^3} \\&  =  \frac{q^9+q^8+q^7}{(1-q^6)(1-q^3)^3}.\end{align*}

Thus,\begin{align*}
    \sum_{n=0}^\infty  e_2p_4(n)q^n & \equiv   \frac{q^{10}+q^8+q^6}{(1-q^3)^2(1-q^6)^2}+ \frac{2(q^9+q^8+q^7)}{(1-q^6)(1-q^3)^3}\pmod 3\\ & = \frac{q^6 + 2 q^7 + 3 q^8 + q^9 + q^{10} - q^{11} - 3 q^{13} - 2 q^{14} - 2 q^{15}}{(1-q^6)^2(1-q^3)^3}
    \\ & \equiv\frac{D(q)}{(1-q^3)^3(1-q^6)^2}\pmod 3,
\end{align*}
where $D(q)=q^6 + 2 q^7 +  q^9 + q^{10} +2 q^{11} + q^{14} + q^{15}$.

Then \begin{align*}\notag \sum_{n=0}^\infty (e_2p_4(n)-e_2p_4(n-54))q^n  & \equiv D(q) \frac{1-q^{54}}{(1-q^3)^3(1-q^6)^2}\pmod 3\\ & \equiv D(q) \frac{(1-q^{3})^9(1+q^{3})^9}{(1-q^3)^3(1-q^6)^2}\pmod 3\\&  \equiv D(q)\frac{(1-q^{3})^6(1+q^{3})^9}{(1-q^6)^2}\pmod 3\\&  \equiv D(q)(1-q^{3})^4(1+q^{3})^7 \pmod 3.\end{align*} Hence, for $n\geq 54$, $e_2p_4(n)-e_2p_4(n-54)\equiv 0 \pmod 3$.
\bigskip

\noindent \textit{Proof of ${e_2p_4(n)\equiv e_2p_4(n+96)\pmod 4}$.}

Let $\rho(\lambda) = \{r_1, r_2, r_3, r_4\}$, the multiset of residues modulo $4$ of the parts of $\lambda$. Here $0\leq r_i\leq 3$, $i=1,2,3,4$, and the order of the residues does not matter. For example, if $\rho(\lambda) = \{3,3,0,2\}$, then $\lambda$ has two parts congruent to $3$ modulo $4$, and one part each congruent to $0$ and to $2$ modulo $4$.

By examining all combinations of residues modulo $4$ for parts of $\lambda\in \mathcal P(4,n)$, we see that 
\begin{itemize}
\item[(i)] $e_2(\lambda)\equiv 1\pmod 4$ if and only if  $\rho(\lambda)$ is one of 
\begin{align*} & \{1,1,1,2\},\,  \{0,0,1,1\}, \,  \{1,1,2,2\}, \,  \{1,1,0,2\}, \, \{1,1,2,3\}, \,   \\  &
\{3,3,3,2\},\, \{0,0,3,3\}, \, \{3,3,2,2\}, \,  \{3,3,0,2\}, \, \{3,3,2,1\}
\end{align*}

\item[(ii)] $e_2(\lambda)\equiv 2\pmod 4$ if and only if $\rho(\lambda)$ is one of \begin{align*}&  \{1,1,1,1\}, \,  \{1,2,2,2\}, \,   \{0,0,1,2\}, \,    \{1,1,3,3\}\\ &  \{3,3,3,3\}, \,  \{3,2,2,2\}, \,  \{0,0,3,2\}\, 
\end{align*}

\item[(iii)] $e_2(\lambda)\equiv 3\pmod 4$ if and only if $\rho(\lambda)$ is one of \begin{align*} & \{0,1,1,1\}, \,  \{0,1,1,3\}, \,  \{0,0,1,3\}, \,   \{0,1,2,3\},\\ &   \{0,3,3,3\},  \,  \{0,3,3,1\},  \, \{2,2,1,3\}, \,  
\end{align*}
\item[(iv)] $e_2(\lambda)\equiv 0\pmod 4$ in all other cases. 
\end{itemize}

Next, we find the generating function for $e_2p_4(n)$ modulo $4$.  Consider for example the case $\lambda$ with $\rho(\lambda)=\{1,1,0,2\}$. 
We can write $\lambda=\mu\cup\eta$, 
where $\mu$ and $\eta$ are each partitions with two parts, the parts of $\mu$ are both congruent to $1\pmod 4$, one part of $\eta$ is congruent to $0$ (so, at least $4)$ and the other is congruent to $2\pmod 4$. Then the conjugate  of $\mu$ is $\mu'=(2^a, 1^b)$, with $a\equiv 1\pmod 4$ and $b\equiv 0\pmod 4$. The generating function for partitions $\lambda$ of $n$ with $\rho(\lambda)=\{1,1,0,2\}$ is given by $$\frac{q^2}{1-q^8}\frac{1}{1-q^4}\frac{q^4}{1-q^4}\frac{q^2}{1-q^4}.$$ 

Similarly, the generating function for partitions $\lambda$ of $n$ with $\rho(\lambda)=\{3,3,0,2\}$ is given by $$\frac{q^6}{1-q^8}\frac{1}{1-q^4}\frac{q^4}{1-q^4}\frac{q^2}{1-q^4}.$$ Together, we obtain  $$\frac{q^8}{(1-q^8)(1-q^4)^3}+ \frac{q^{12}}{(1-q^8)(1-q^4)^3}=\frac{q^8(1+q^4)}{(1-q^8)(1-q^4)^3}.$$ 
\medskip

Continuing with the rest of the residue multisets in cases i, ii, and iii and combining the pairs in the columns above,  we have
\begin{align*}\sum_{n=0}^\infty e_2p_4(n)q^n  \equiv & 
    \frac{q^5(1+q^6)}{(1-q^{12})(1-q^8)(1-q^4)^2} + \frac{q^{10}(1+q^4)}{(1-q^8)^2(1-q^4)^2} \\ &  +\frac{q^{6}(1+q^4)}{(1-q^8)^2(1-q^4)^2}  + \frac{q^8(1+q^4)}{(1-q^8)(1-q^4)^3}+\frac{q^7(1+q^2)}{(1-q^8)(1-q^4)^3}\\ & 
+2\left(  \frac{q^4(1+q^8)}{(1-q^{16})(1-q^{12})(1-q^8)(1-q^4)} + \frac{q^{7}(1+q^2)}{(1-q^{12})(1-q^8)(1-q^4)^2} \right.\\ & \left. \ \ \ \ \ \  +\frac{q^{11}(1+q^2)}{(1-q^8)(1-q^4)^3}+ \frac{q^8}{(1-q^8)^2(1-q^4)^2}\right)\\ & 
+3\left(
    \frac{q^7(1+q^6)}{(1-q^{12})(1-q^8)(1-q^4)^2} + \frac{q^{9}(1+q^2)}{(1-q^8)(1-q^4)^3} \right.\\ & \left.\ \ \ \ \ \  +\frac{q^{8}(1+q^4)}{(1-q^8)(1-q^4)^3} +\frac{q^{10}}{(1-q^4)^4}\right) \pmod 4.
\end{align*}

To prove  $$ \sum_{n=0}^\infty (e_2p_4(n)-e_2p_4(n-96))q^n\equiv 0 \pmod 4,$$ we show that  \begin{equation}\label{C(q)}(1-q^{96})\sum_{n=0}^\infty e_2p_4(n)q^n \equiv C(q)\pmod 4,\end{equation} where $C(q)$ is a polynomial of degree at most $95$.

After a computer calculation, we found that
$$\displaystyle(1-q^{96})\sum_{n=0}^\infty e_2p_4(n)q^n \equiv \frac{1}{(1 - q^4)^3}q^4f(q)\pmod 4,$$ where 
 $f(q)\equiv f_0(q)+f_1(q)+f_2(q)+f_3(q)\pmod 4$, with
    \smallskip

    $f_0(q)=\sum_{i=0}^{24}a_iq^{4i}$ with $a_{23}=a_{24}=0$ and $$a_i=\begin{cases}2 & \text{ if } i\equiv 0,1\mod 3,\\ 0 & \text{ if } i\equiv -1\mod 3,\end{cases}$$

    $f_1(q)=q\sum_{i=0}^{24}b_iq^{4i}$ with $b_0=1, b_{24}=2$ and for $1\leq i\leq 23$, $$b_i=\begin{cases}-1 & \text{ if } i\equiv 0\mod 3,\\ 1 & \text{ if } i\equiv 1\mod 3,\\0 & \text{ if } i\equiv -1\mod 3,\end{cases}$$

$f_2(q)=q^2\sum_{i=0}^{24}c_iq^{4i}$, with $c_0=1$, $c_{24}=-1$ and for $1\leq i\leq 23$, $c_i=0$,\medskip

$f_3(q)=q^3\sum_{i=0}^{24}d_iq^{4i}$ with $d_{24}=0$ and for $0\leq i\leq 23$,$$d_i=\begin{cases}1 & \text{ if } i\equiv \pm 1\mod 3,\\ 2 & \text{ if } i\equiv 0\mod 3.\end{cases}$$

     Then, since $\displaystyle \frac{1}{(1-q^4)^3}=\sum_{k=0}^\infty\binom{k+2}{2}q^{4k}$, we have  \begin{align*}& \frac{1}{(1 - q^4)^3} q^4f(q) \equiv  \\ &  q^4\left(h(q) +\sum_{k=0}^\infty\left(\sum_{i=0}^{24}(a_i+qb_i+q^2c_i+q^3d_i)\binom{k+2+24-i}{2}\right)q^{4(k+24)}\right)\pmod 4,\end{align*} where \begin{align*}h(q) =& 2+q+q^2+2 q^3+8 q^4+4 q^5+3 q^6+7 q^7+18 q^8+9 q^9+6 q^{10}+16 q^{11}\\ & +34 q^{12} +15 q^{13}+10 q^{14}+31 q^{15}+58 q^{16}+23 q^{17}+15 q^{18}+53 q^{19}+90 q^{20}\\ & +33 q^{21} +21 q^{22} +83 q^{23}+132 q^{24}+44 q^{25}+28 q^{26}+123 q^{27}+186 q^{28}+57 q^{29}\\ & +36 q^{30} +174 q^{31} +252 q^{32}+72 q^{33}+45 q^{34}+237 q^{35}+332 q^{36}+88 q^{37}\\ & +55 q^{38}+314 q^{39} +428 q^{40}+106 q^{41}+66 q^{42}+406 q^{43}+540 q^{44}+126 q^{45}\\ & +78 q^{46}+514 q^{47}+670 q^{48} +147 q^{49} +91 q^{50}+640 q^{51}+820 q^{52}+170 q^{53}\\ & +105 q^{54}+785 q^{55}+990 q^{56}+195 q^{57} +120 q^{58}+950 q^{59}+1182 q^{60}+221 q^{61}\\ & +136 q^{62}+1137 q^{63}+1398 q^{64}+249 q^{65} +153 q^{66}+1347 q^{67}+1638 q^{68}\\ & +279 q^{69}+171 q^{70}+1581 q^{71}+1904 q^{72}+310 q^{73} +190 q^{74}+1841 q^{75}\\ & +2198 q^{76}+343 q^{77}+210 q^{78}+2128 q^{79}+2520 q^{80}+378 q^{81} +231 q^{82}\\ & +2443 q^{83}+2872 q^{84}+414 q^{85}+253 q^{86}+2788 q^{87}+3256 q^{88}+452 q^{89}\\ & +276 q^{90}+3164 q^{91}+3672 q^{92}+492 q^{93}+300 q^{94}+3572 q^{95}.\end{align*}  
     After reducing coefficients modulo $4$, the polynomial $q^4h(q)$ has degree   $90$. 
     
      For the expression involving binomials, we have \begin{align*}\sum_{i=0}^{24}a_i\binom{k+2+24-i}{2}& =8 (515 + 58 k + 2 k^2),\\   
\sum_{i=0}^{24}b_i\binom{k+2+24-i}{2}& =2 (268 + 23 k + k^2)\equiv 0 \pmod 4,
\\ \sum_{i=0}^{24}c_i\binom{k+2+24-i}{2}& =12 (27 + 2 k),\\ \sum_{i=0}^{24}d_i\binom{k+2+24-i}{2}& =4 (1003 + 114 k + 4 k^2).
\end{align*} 
This proves \eqref{C(q)} and completes the proof of the theorem.  
\qed

\section{Partitions associated with $e_j(\lambda)$}
Given a partition $\lambda=(\lambda_1, \lambda_2, \ldots, \lambda_\ell)$, we define $\pre_j(\lambda)$ to be the partition whose multiset of parts is  $$\{\lambda_{i_1}\lambda_{i_2}\cdots \lambda_{i_i}\mid 1\leq i_1<i_2<\cdots <i_j\leq\ell\}, \ \ \text{if } \ell\geq j.$$ If $\ell(\lambda)<j$, we set $\pre_j(\lambda):=\emptyset$. Thus, $\pre_j(\lambda)$ is the partition whose parts are the monomials  in $e_j(\lambda)$. If $\mathcal A(n) \subseteq \mathcal P(n)$, we denote by $\pre_j(\mathcal A(n))$ the set of partitions $\pre_j(\lambda)$ with $\lambda\in \mathcal A(n)$. 

We denote by $o_i(n)$ (respectively $d_i(n)$) the number of partitions in $\pre_i(\mathcal P(n))$ with odd (respectively distinct) parts. Since $\pre_1(\mathcal P(n))=\mathcal P(n)$, Euler's theorem states that $o_1(n)=d_1(n)$. 

Based on numerical evidence, we make the following conjecture. 
\begin{conjecture}We have

\begin{itemize}
\item[(i)] $o_2(n)\geq d_2(n)$  for $n\geq 0$,
\item[(ii)]$o_3(n)\geq d_3(n)$  for $n\geq 21$.
\end{itemize}
\end{conjecture}

The first part of the conjecture follows from the conjectural injectivity of the mapping $\pre_2$. 
\begin{conjecture}\label{L_pre2}
    If $\lambda, \mu$ are partitions of $n$ for some $n\geq 0$, then $\pre_2(\lambda)=\pre_2(\mu)$ if and only if $\lambda=\mu$.
\end{conjecture}
More generally, we also conjecture the following.
\begin{conjecture}\label{L_prek}
    For $j\geq 3$, the mapping $\pre_j$ is injective on the set of partitions of $n$ of length greater than $j$.
\end{conjecture}

Now, suppose Conjecture \ref{L_pre2} is true. 
Let $n\geq 0$. Denote by $\mathcal O_j(n)$ (respectively $\mathcal D_j(n)$) the set of  partitions in $\pre_j(\mathcal P(n))$ with odd (respectively distinct) parts.  If $\pre_2(\lambda)\in \mathcal O_2(n)$, then $\lambda\in \mathcal O_1(n)\subseteq \mathcal P(n)$, that is, $\lambda$ has odd parts or $\ell(\lambda)=1$. 
Then, by Conjecture \ref{L_pre2}, $o_2(n)=o_1(n)$ if $n$ odd and $o_2(n)=o_1(n)+1$ if $n$ even.

Clearly, if $\lambda$ has a repeated part and $\ell(\lambda)>2$, then  $\pre_2(\lambda)\not\in \mathcal D_2(n)$. If $n$ is even, then $\pre_2((n/2,n/2))=(n^2/4)\in \mathcal D_2(n)$. We have $d_2(0)=d_1(0)=1$.   If $n$ is odd, $d_2(n)\leq d_1(n)$. If $n>0$ is even $d_2(n)\leq d_1(n)+1$. 

Hence $o_2(n)\geq d_2(n)$ for all $n\geq 0$. 

Next we prove the injectivity of the mapping $\pre_2$ on the set of partitions of $n$ of length at most $3.$

\begin{remark} In general, $|\lambda|\neq |\pre_i(\lambda)|$. Clearly, $\pre_j(\lambda)=\emptyset$ if and only if $\ell(\lambda)<j$. Moreover, $\ell(\pre_j(\lambda))=\displaystyle\binom{\ell(\lambda)}{j}$. 
\end{remark}
\begin{proposition}
   If $\lambda, \mu$ are partitions of $n$ for some $n\geq 0$ and $\ell(\lambda)$, $\ell(\mu)\leq 3$, then $\pre_2(\lambda)=\pre_2(\mu)$ if and only if $\lambda=\mu$. 
\end{proposition}

\begin{proof}  Let $\lambda, \mu$ be partitions of $n$.  By definition, if $\lambda=\mu$, then $\pre_2(\lambda)=\pre_2(\mu)$. Next, assume $\pre_2(\lambda)=\pre_2(\mu)\neq\emptyset$. Then $\ell(\lambda)=\ell(\mu)=:\ell$. 

\noindent \underline{Case $\ell=1$.} Clearly  $\pre_2(\lambda)=\emptyset$ if and only if  $\lambda=(n)$, so $\lambda = \mu$.  

\noindent \underline{Case $\ell=2$.}
The partitions of $n$ into two parts are $\mathcal P(2,n)=\{(n-1,1), (n-2,2), (n-3,3), \ldots, (\lceil n/2 \rceil, \lfloor n/2 \rfloor)\}$, and so $$\pre_2(\mathcal P(2,n)) = \{ (n-1), (2(n-2)), (3(n-3)), \ldots, (\lceil n/2 \rceil  \lfloor n/2 \rfloor)\}.$$ The sizes of those partitions are strictly increasing and hence distinct. It follows that $\lambda=\mu$.

\noindent \underline{Case $\ell=3$.} By the order relations between parts, in any partition $\lambda$ with $\ell(\lambda)=3$ we have $\lambda_1\lambda_2\geq \lambda_1\lambda_3\geq\lambda_2\lambda_3$. Thus $\lambda_i\lambda_j=\mu_i\mu_j$ if $1\leq i<j\leq 3$.

Since $\lambda_1+\lambda_2 +\lambda_3=\mu_1+\mu_2+\mu_3$, we have \begin{align*}\lambda_1\mu_3+\lambda_2\mu_3 +\lambda_3\mu_3& =\mu_1\mu_3+\mu_2\mu_3 +\mu_3\mu_3 \\ & =\lambda_1\lambda_3+\lambda_2\lambda_3+\mu_3\mu_3,\\ (\mu_3-\lambda_3)(\lambda_1+\lambda_2 -\mu_3)& = 0.  
\end{align*}
Since $\lambda_1\geq\mu_3$ and $\lambda_2>0$, we must have $\mu_3=\lambda_3$. Then from $\lambda_i\lambda_3=\mu_i\mu_3$ for $i=1,2$, it follows that $\lambda=\mu$. 
\end{proof}

\begin{remark}
    If $\ell(\lambda)\geq 4$, the partition $\lambda$ determines a partial order on the parts of $\pre_2(\lambda)$. Understanding the lattice of parts of $\pre_2(\lambda)$ may lead to the proof of the injectivity of $\pre_2$ on $\mathcal P(n)$.
\end{remark}

Let $\mathcal B(n)$ be the set of binary partitions of $n$. Recall that  $B(n)=|\mathcal B(n)|$.  We prove the injectivity of $\pre_2$ on the set $\mathcal B(n)$. 

\begin{proposition} Let $\lambda, \mu\in \mathcal B(n)$. Then $\pre_2(\lambda)=\pre_2(\mu)$ if and only if $\lambda=\mu$.
\end{proposition}
\begin{proof} 
Let $\lambda \in \mathcal B(n)$ and set $\nu:=\pre_2(\lambda)$.  Then  $\nu$ is a binary partition. We show that $\nu$ completely determines $\lambda$. 

\noindent \underline{Case 1:}  $m_\nu(1) \neq 0$.  

Since $m_\nu(1)=\binom{m_\lambda(1)}{2}$, $m_\nu(1)$ determines $m_\lambda(1)$. 
Since $m_\nu(2)=m_\lambda(1)m_\lambda(2)$, $m_\lambda(2)$ is determined by $m_\nu(1)$ and $m_\nu(2)$. 
 
 We proceed inductively. For $k\geq 2$, we have 
 $$m_\nu(2^k)=\begin{cases} \displaystyle  \binom{m_\lambda(2^{ k/2})}{2}+\sum_{i=0}^{k/2-1}m_\lambda(2^{k-i})m_\lambda(2^i)  &  \text{ if $k$ is even},\\  \ \\  \displaystyle\sum_{i=0}^{\lfloor k/2\rfloor}m_\lambda(2^{k-i})m_\lambda(2^i)  & \text{ if $k$ is odd}. \end{cases}$$ This determines $m_\lambda(2^k)$ from $m_\nu(2^k)$ and $m_\lambda(2^j)$ for $0\leq j\leq k-1$. 

 \noindent \underline{Case 2:} $m_\nu(1) = 0$. 

 If $n$ is odd, then $m_\lambda(1)=1$ and the argument above show that $\nu$ determines $\lambda$. 

 If $n$ is even, then $m_\lambda(1)=0$. We have $\lambda/2\in \mathcal B(n/2)$ and  $\pre_2(\lambda/2)=\nu/4$. We repeat the process until  $\lambda/2^i$ has a part equal to $1$. Then $\pre_2(\lambda/2^i)=\nu/{4^i}$ completely determines $\lambda/2^i$ and thus $\nu$ determines $\lambda$. 
 \end{proof}

Next, we prove a recurrence for the number $b_{1,2}(n)$ of parts equal to $1$ in all partitions in $\pre_2(\mathcal B(n))$.

\begin{theorem} \label{rec_b12} The sequence $b_{1,2}(n)$ satisfies the following recurrence. If $n\geq 3$, $$b_{1,2}(n)=b_{1,2}(n-1)+b_{1,2}(\lfloor (n+1)/2\rfloor)+ b_{1,2}(\lceil (n+1)/2 \rceil),$$ and $b_{1,2}(1)=0$, $b_{1,2}(2)=1$.
    \end{theorem}

    \begin{proof}
    We first consider the  generating function for a sequence $a(n)$ defined by the recurrence  $$a(n)=a(n-1)+a(\lfloor (n+1)/2\rfloor)+a(\lceil (n+1)/2 \rceil), \ \ n\geq 3,$$ and $a(1)=0, \, a(2)=1$. This is sequence A131205 in \cite{OEIS} shifted by one. Let $$F(q):=\sum_{n=1}^\infty a(n)q^n.$$ Then 
\begin{align*}
F(q)& =q^2+\sum_{n=3}^\infty a(n)q^n\\ & = q^2+ \sum_{n=3}^\infty a(n-1)q^n + \sum_{n=3}^\infty a(\lfloor (n+1)/2\rfloor)q^n
   +  \sum_{n=3}^\infty a(\lceil (n+1)/2 \rceil)q^n\\ & = q^2+qF(q)+F(q^2)+\frac{1}{q}F(q^2)+\frac{1}{q}F(q^2)+\frac{1}{q^2}\big(F(q^2)-q^4\big).
\end{align*}Hence $$F(q)=\frac{(1+q)^2}{q^2(1-q)}F(q^2).$$ 
Now let
$$G(q):=\frac{q^2}{1-q} \prod_{i=0}^{\infty} \frac{1+q^{2^i}}{1-q^{2^i}}.$$
We can write
\begin{align*}
G(q) &=\frac{q^2}{1-q} \frac{1+q}{1-q}\prod_{i=1}^{\infty} \frac{1+q^{2^i}}{1-q^{2^i}}\\
&=\frac{q^2}{1-q} \frac{1+q}{1-q} \prod_{i=0}^{\infty} \frac{1+q^{2^{i+1}}}{1-q^{2^{i+1}}}\\
&=\frac{q^2}{1-q} \frac{1+q}{1-q} \frac{1-q^2} {q^4} \frac{q^4}{1-q^2} \prod_{i=0}^{\infty} \frac{1+q^{2^{i+1}}}{1-q^{2^{i+1}}}\\
&=\frac{(1+q)^2}{q^2(1-q)}G(q^2).
\end{align*}
Therefore, $F(q)=G(q).$

Next, recall that $\pre_2$ is injective on $\mathcal B(n)$. As previously mentioned,   if $\lambda\in \mathcal B(n)$ and $\nu=\pre_2(\lambda)$, then $m_\nu(1)=\binom{m_\lambda(1)}{2}$.

As in the proof of Theorem \ref{Th3}, the number of partitions $\lambda\in \mathcal B(n)$ with $m_\lambda(1)=k$ equals $B(\frac{n-k}{2})$, so  $$b_{1,2}(n)=\sum_{k=0}^\infty\binom{k}{2}B\left(\frac{n-k}{2}\right).$$

The generating function for $ B(n)$ is
$$ \sum_{n=0}^\infty  B(n)\,q^n = \prod_{i=0}^\infty \frac{1}{1-q^{2^i}} = (1-q) \prod_{i=0}^\infty \frac{1+q^{2^i}}{1-q^{2^i}}.
$$
The generating function for $\binom{n}{2}$ is
$$
\sum_{n=0}^\infty \binom{n}{2}\,q^n = \frac{q^2}{(1-q)^3}.
$$
For $n=2m$, 
$$\sum_{k=0}^\infty\binom{k}{2}B\left(\frac{n-k}{2}\right)
=\sum_{k=0}^\infty\binom{k}{2}B\left(m-\frac{k}{2}\right)
=\sum_{k=0}^\infty\binom{2k}{2}B\left(m-k\right).
$$
For $n=2m+1$, 
$$\sum_{k=0}^\infty\binom{k}{2}B\left(\frac{n-k}{2}\right)
=\sum_{k=0}^\infty\binom{k}{2}B\left(m-\frac{k-1}{2}\right)
=\sum_{k=0}^\infty\binom{2k+1}{2}B\left(m-k\right).
$$
The generating function for $\binom{2n}{2}$ is
$$
\sum_{n=0}^\infty \binom{2n}{2}\,q^n = \frac{q(1+3q)}{(1-q)^3}.
$$
The generating function for $\binom{2n+1}{2}$ is
$$
\sum_{n=0}^\infty \binom{2n+1}{2}\,q^n = \frac{q(3+q)}{(1-q)^3}.
$$
We deduce that  the generating function for 
$$b_{1,2}(n)=\sum_{k=0}^\infty\binom{k}{2}B\left(\frac{n-k}{2}\right)$$
is given by 
\begin{align*}
 \sum_{n=0}^\infty b_{1,2}(n)q^n &=  \left(\frac{q^2(1+3q^2)}{(1-q^2)^2}+\frac{q^3(3+q^2)}{(1-q^2)^2} \right) \prod_{i=0}^\infty \frac{1+q^{2^{i+1}}}{1-q^{2^{i+1}}}\\ & 
=\frac{q^2(1+q)}{(1-q)^2} \prod_{i=0}^\infty \frac{1+q^{2^{i+1}}}{1-q^{2^{i+1}}}\\
& = \frac{q^2(1+q)}{(1-q)^2} \prod_{i=1}^\infty \frac{1+q^{2^{i}}}{1-q^{2^{i}}} \\ & 
= \frac{q^2(1+q)}{(1-q)^2} \frac{1-q}{1+q} \prod_{i=0}^\infty \frac{1+q^{2^{i}}}{1-q^{2^{i}}}\\
& = \frac{q^2}{1-q} \prod_{i=0}^\infty \frac{1+q^{2^{i}}}{1-q^{2^{i}}}=G(q).
\end{align*}
\end{proof}

Next we consider $b_{2,2}(n)$, the number of parts equal to $2$ in all partitions in $\pre_2(\mathcal B(n))$.
For $n\geq 1$, set $\Delta b_{2,2}(n) := b_{2,2}(n+1)-b_{2,2}(n)$.

\begin{theorem} \label{rec_b22} 
For $n\geq 0$,  we have 
$\Delta b_{2,2}(2n+1) = \Delta b_{2,2}(2n) = b_{1,2}(n+1)$.
    \end{theorem}
\begin{proof}
We first define two bijections, $\varphi$ and $\psi$. Let $n\geq 1$ and define $\varphi: \mathcal B(2n)\to \mathcal B(2n+1)$  by $\varphi(\lambda)=\lambda\cup (1)$.  Clearly, $\varphi$ is a bijection and $m_\lambda(2)=m_{\varphi(\lambda)}(2)$ for all $\lambda\in \mathcal B(2n)$. If $\lambda\in \mathcal B(2n)$,  then $\pre_2(\lambda)$ has $m_\lambda(1)m_\lambda(2)$ parts equal to $2$ and, by construction,  
$\pre_2(\varphi(\lambda))=\pre_2(\lambda\cup (1))$ has $(m_\lambda(1)+1)m_\lambda(2)$ parts equal to $2$. Hence, 
 \begin{equation}\label{diff b22} b_{2,2}(2n+1)-b_{2,2}(2n)=\sum_{\lambda\in \mathcal B(2n)}m_\lambda(2).\end{equation} The right hand side of \eqref{diff b22} is the total number of parts equal to $2$ in all partitions in $\mathcal B(2n)$.

Next, note that if $\lambda\in B(2n)$ has no parts equal to $1$, then $\pre_2(\lambda)$ has no parts equal to $2$. Denote by $\mathcal B^*(2n)$ the set of partitions in $\mathcal B(2n)$ that have at least one part equal to $1$.  Then,  for $\lambda\in \mathcal B^*(2n)$, $m_\lambda(1)$ is even. 
Let  $\psi: \mathcal B^*(2n)\to \mathcal B(2n-1)$ be defined by $\psi(\lambda)=\lambda\setminus (1)$. Clearly, $\psi$ is a bijection and $m_\lambda(2)=m_{\psi(\lambda)}(2)$ for all $\lambda\in \mathcal B^*(2n)$. If $\lambda\in \mathcal B^*(2n)$,  then $\pre_2(\lambda)$ has $m_\lambda(1)m_\lambda(2)$ parts equal to $2$ and, by construction,  
$\pre_2(\psi(\lambda))=\pre_2(\lambda\setminus (1))$ has $(m_\lambda(1)-1)m_\lambda(2)>0$ parts equal to $2$.  Hence,  \begin{equation*}\label{diff b22*} b_{2,2}(2n)-b_{2,2}(2n-1)=\sum_{\lambda\in \mathcal B^*(2n)}m_\lambda(2)=\sum_{\lambda\in \mathcal B(2n-1)}m_\lambda(2),\end{equation*}
which with \eqref{diff b22} implies that for $n\geq 1$, 
  \begin{equation*} \Delta_{2,2}(n)=b_{2,2}(n+1)-b_{2,2}(n)=\sum_{\lambda\in \mathcal B(n)}m_\lambda(2).\end{equation*}

 The generating function for the number of parts equal to $2$ in all partitions in $\mathcal B(n)$ is obtained by 
taking the derivative   of $$\frac{1-q^2}{1-zq^{2}}\prod_{i=0}^\infty\frac{1}{1-q^{2^i}}$$  with respect to $z$ and evaluating at $z=1$.

Thus $$\sum_{n=0}^\infty \Delta_{2,2}(n)q^n=\frac{q^2}{1-q^{2}}\prod_{i=0}^\infty\frac{1}{1-q^{2^i}}.$$

Moreover, $$\sum_{n=0}^\infty (-1)^n\Delta_{2,2}(n)q^n=\frac{q^2}{1-q^{2}}\frac{1}{1+q}\prod_{i=1}^\infty\frac{1}{1-q^{2^i}}.$$ Since $$\sum_{n=0}^\infty \Delta_{2,2}(n)q^n= \sum_{n=0}^\infty \Delta_{2,2}(2n)q^{2n}+\sum_{n=0}^\infty \Delta_{2,2}(2n+1)q^{2n+1}$$ and $$\sum_{n=0}^\infty (-1)^n\Delta_{2,2}(n)q^n= \sum_{n=0}^\infty \Delta_{2,2}(2n)q^{2n}-\sum_{n=0}^\infty \Delta_{2,2}(2n+1)q^{2n+1},$$ we have \begin{align*}\sum_{n=0}^\infty \Delta_{2,2}(2n)q^{2n}& =\frac{1}{2} \frac{q^2}{1-q^{2}}\left(\frac{1}{1-q}+\frac{1}{1+q} \right)\prod_{i=1}^\infty\frac{1}{1-q^{2^i}}\\\ & =  \frac{q^2}{(1-q^{2})^2}\prod_{i=1}^\infty\frac{1}{1-q^{2^i}}. \end{align*} 
Replacing $q$ by $q^{1/2}$, we obtain 
\begin{align*}\sum_{n=0}^\infty \Delta_{2,2}(2n)q^{n}& = \frac{q}{(1-q)^2}\prod_{i=0}^\infty\frac{1}{1-q^{2^i}}=\frac{1}{q}\sum_{n=0}^\infty b_{1,2}(n)q^n=\sum_{n=0}^\infty b_{1,2}(n+1)q^n.\end{align*}  In the last equality we used $b_{1,2}(0)=0$.

Similarly, \begin{align*}\sum_{n=0}^\infty \Delta_{2,2}(2n+1)q^{2n+1}& =  \frac{q^3}{(1-q^{2})^2}\prod_{i=1}^\infty\frac{1}{1-q^{2^i}}. \end{align*} Dividing by $q$ and replacing $q$ by $q^{1/2}$, we obtain \begin{align*}\sum_{n=0}^\infty \Delta_{2,2}(2n+1)q^{n}& = \sum_{n=0}^\infty b_{1,2}(n+1)q^n.\end{align*}

\end{proof}

\section{Conjectures on the log-concavity of $e_j(n)$ and related functions }

 A sequence $\{a_n\}$ is log-concave if $$a_n^2\geq a_{n-1}a_{n+1} \text{ for all } n.$$
In \cite{DP}, DeSalvo and Pak proved the log-concavity of the partition function $p(n)$ for $n> 25$. Numerical evidence suggests that  the function $e_jp(n)$ is log-concave for $j>2$, and also for  $j= 1,2$ if $n$ is sufficiently large. 
\begin{conjecture}
\begin{align*}e_1p(n)^2  \geq e_1p(n-1)\cdot e_1p(n+1), & \text{ \ for } n>23; \\ e_2p(n)^2 \geq e_2p(n-1)\cdot e_2p(n+1),& \text{ \  for } n>17;\\ e_jp(n)^2 \geq  e_jp(n-1)\cdot e_jp(n+1),& \text{ \  for } j>2, n>0.
    \end{align*}
\end{conjecture}

Moreover, numerical examples suggest that the functions $F(n):=e_jp(n)/p(n)$ and $F(n)/n$ are also log-concave. 

\begin{conjecture}
   For $j,n>0$,
   \begin{enumerate}
      \item[(i)] $\displaystyle{  \left( \frac{e_jp(n)}{p(n)} \right)^2 \geq  \frac{e_jp(n-1)}{p(n-1)} \cdot \frac{e_jp(n+1)}{p(n+1)}  }$,\\
       \item[(ii)] $\displaystyle{ \left( \frac{e_jp(n)}{n\,p(n)} \right)^2 \geq  \frac{e_jp(n-1)}{(n-1)\,p(n-1)} \cdot \frac{e_jp(n+1)}{(n+1)\,p(n+1)}  }$.
     \end{enumerate}
\end{conjecture}

\end{document}